\documentclass[10pt]{amsart}
\usepackage{amssymb}
\usepackage{bm}
\usepackage{graphicx}
\usepackage[centertags]{amsmath}
\usepackage{amsfonts}
\usepackage{amsthm}
\newtheorem{thm}{Theorem}
\newtheorem{cor}[thm]{Corollary}
\newtheorem{lem}[thm]{Lemma}

\newtheorem{prop}[thm]{Proposition}
\newtheorem{claim}[thm]{Claim}
\newtheorem{fact}[thm]{Fact}

\theoremstyle{definition}
\newtheorem{rem}{Remark}

\newcommand{\nn}{\mathbb{N}}
\newcommand{\ee}{\varepsilon}

\newcommand{\hhh}{\mathcal{H}}
\newcommand{\uuu}{\mathcal{U}}
\newcommand{\vvv}{\mathcal{V}}

\newcommand{\ccc}{\mathcal{C}}
\newcommand{\SB}{\mathbf{\Sigma}}
\newcommand{\PB}{\mathbf{\Pi}}

\newcommand{\inte}{\mathrm{Int}}

\begin{document}

\title{The Steinhaus property and Haar-null sets}
\author{Pandelis Dodos}
\address{Texas A$\&$M University, Department of Mathematics, 3368 TAMU, College Station,
TX 77843-3368.}
\email{pdodos@math.ntua.gr}

\footnotetext[1]{2000 \textit{Mathematics Subject Classification}: 54H11, 28C10.}
\footnotetext[2]{\textit{Key words}: Haar-null sets, Polish groups,
Steinhaus Theorem, Borel measures.}

\maketitle


\begin{abstract}
It is shown that if $G$ is an uncountable Polish group and $A\subseteq G$ is a universally
measurable set such that $A^{-1}A$ is meager, then the set $T_l(A)=\{\mu\in P(G): \mu(gA)=0
\text{ for all } g\in G\}$ is co-meager. In particular, if $A$ is analytic and not left
Haar-null, then $1\in\inte(A^{-1}AA^{-1}A)$.
\end{abstract}


\section{Introduction}

The purpose of this paper is to show that there exists a satisfactory extension of the classical
Steinhaus Theorem for an arbitrary Polish group. In order to get the extension one needs, first,
to isolate the appropriate $\sigma$-ideal on which the result will be applied. For the class of
abelian Polish groups this is the $\sigma$-ideal of Haar-null sets, defined by J. P. R.
Christensen \cite{Chr}. However, in non-abelian (and non-locally-compact) Polish groups this
$\sigma$-ideal is no longer well-behaved. Actually, by the results of S. Solecki in \cite{S2},
the Steinhaus property of Haar-null sets fails in ``most" non-abelian Polish groups. Notice also
that the conclusion of the Steinhaus Theorem is rather strong. If $A\subseteq\mathbb{R}$ is
of positive Lebesgue measure, then $A-A$ contains a neighborhood of $0$. If we relax the
conclusion to $A-A$ is not meager\footnote[3]{We recall that a subset $A$ of a topological
space $X$ is said to be \textit{meager} (or \textit{of first category}) if $A$ is covered
by a countable union of closed nowhere dense sets. The complement of a meager set is usually
referred as \textit{co-meager}.}, then this is valid in every abelian Polish group.

To state our result we need some definitions. Let $G$ be a Polish group and $A\subseteq G$ be
a universally measurable set. The set $A$ is said \textit{Haar-null} if there exists $\mu\in
P(G)$ (i.e. $\mu$ is a Borel probability measure on $G$) such that $\mu(g_1Ag_2)=0$ for all
$g_1, g_2\in G$. It is said to be \textit{left Haar-null} if there exists $\mu\in P(G)$ such
that $\mu(gA)=0$ for all $g\in G$. By the results in \cite{ST} and \cite{S2}, the notions of
Haar-null and left Haar-null set are distinct (however, they obviously agree on abelian groups).
We let
\[ T(A)=\{ \mu\in P(G): \mu(g_1Ag_2)=0 \text{ for all } g_1,g_2\in G\} \]
and
\[ T_l(A)=\{ \mu\in P(G): \mu(gA)=0 \text{ for all } g\in G\}. \]
It is easy to see that if $A$ is analytic\footnote[4]{We recall that a subset $A$ of a Polish
space $X$ is said to be \textit{analytic} if there exists a continuous map $f:\nn^\nn\to X$ with
$f(\nn^\nn)=A$. It is a classical result that every Borel subset of a Polish space is analytic.
It is also well-known that an analytic set which is not meager is actually co-meager in a
non-empty open set.}, then both $T(A)$ and $T_l(A)$ are faces (i.e. extreme convex subsets) of
$P(G)$ with the Baire property. It follows by \cite[Theorem 4]{D2} that the sets $T(A)$ and
$T_l(A)$ are either meager, or co-meager. A set $A$ is said to be \textit{generically Haar-null}
if $T(A)$ is co-meager. Respectively, the set $A$ is said to be \textit{generically left
Haar-null} if $T_l(A)$ is co-meager.

For every Polish group $G$ the class of generically left Haar-null subsets of $G$ forms a
$\sigma$-ideal. Notice that if $A$ is not generically left Haar-null, then $A$ should not be
considered as a small set (it is null only for a relatively small set of measures). This is
indeed true as the following theorem demonstrates.
\medskip

\noindent \textbf{Theorem A.} \textit{Let $G$ be an uncountable Polish group and $A$ be a
universally measurable subset of $G$. Assume that $A^{-1}A$ is meager. Then $T_l(A)$ is
co-meager. }

\textit{Thus, if $A$ is analytic and not generically left Haar-null (in particular, not left Haar-null), then $A^{-1}A$ is non-meager.}
\medskip

The locally-compact abelian case of Theorem A can be also derived by the results of M. Laczkovich
in \cite{La}, who proved that if $A$ is not covered by an $F_\sigma$ Haar-measure zero set, then
$A^{-1}A$ is co-meager in a neighborhood of the identity. To see that this implies Theorem A,
one invokes \cite[Proposition 5]{D1} which states that if $G$ is locally-compact and
$A\subseteq G$ is covered by an $F_\sigma$ Haar-null set, then $T_l(A)$ is co-meager. Both M.
Laczkovich's result as well as the result of J. P. R. Christensen \cite{Chr} that Haar-null sets
satisfy the Steinhaus property in abelian Polish groups, are heavily depended on the classical
Steinhaus Theorem. The proof of Theorem A follows quite different arguments. It is based on the
fact that if $\hhh$ is a dense $G_\delta$ and hereditary subset of $K(G)$, then this is witnessed
in the probabilities of $G$.

\subsection{Preliminaries} Our general notation and terminology follows \cite{Kechris}. By
$\nn=\{0,1,2,...\}$ we denote the natural numbers. For any Polish space $X$ by $K(X)$ we denote
the hyperspace of all compact subsets of $X$ with the Vietoris topology and by $P(X)$ the space
of all Borel probability measures on $X$ with the weak* topology. Both are Polish (see
\cite{Kechris}). If $d$ is a compatible complete metric of $X$, then by $d_H$ we denote the
Hausdorff metric on $K(X)$ associated to $d$, defined by
\[ d_H(K,C)=\inf\{ \ee>0: K\subseteq C_\ee \text{ and } C\subseteq K_\ee\} \]
where $A_\ee=\{x\in X:d(x,A)\leq\ee\}$ for every $A\subseteq X$. All balls in $K(X)$ are taken
with respect to $d_H$ and are denoted by $B_H$. In $P(X)$ we consider the so called L\'{e}vy
metric $\rho$, defined by
\begin{eqnarray*}
\rho(\mu,\nu)=
\inf\big\{ \ee>0 & : & \mu(A)\leq \nu(A_\ee)+\ee \text{ and } \nu(A)\leq\mu(A_\ee) +\ee \\
& & \text{for every compact (or Borel) subset } A \text{ of } X \big\}
\end{eqnarray*}
(see \cite{BL} for more details). All balls in $P(X)$ are taken with respect to $\rho$ and are
denoted by $B_P$. If $G$ is a Polish group and $\mu, \nu\in P(G)$, then by $\mu*\nu$ we denote
their convolution, defined by
\[ \mu*\nu(A)=\int_G \mu(Ax^{-1})d\nu(x). \]
A subset $\hhh$ of $K(X)$ is said to be \textit{hereditary} if for every $K\in\hhh$ and every
$C\in K(X)$ with $C\subseteq K$ we have that $C\in\hhh$. All the other pieces of notation we use
are standard.


\section{Hereditary, dense $G_\delta$ sets and measures}

Throughout this section $X$ will be a Polish space and $\hhh$ a hereditary, dense $G_\delta$
subset of $K(X)$. By $d$ we denote a compatible complete metric of $X$.
\begin{lem}
\label{l1} Let $X$ and $\hhh$ as above. Then there exists a sequence $(\uuu_n)$ of open, dense
and hereditary subsets of $K(X)$ such that $\hhh=\bigcap_n \uuu_n$.
\end{lem}
\begin{proof} Write $\hhh=\bigcap_n \vvv_n$ where each $\vvv_n$ is open and dense but not
necessarily hereditary. Fix $n$ and define
\[ \ccc_n=\{ K\in K(X): \exists C\subseteq K \text{ compact with } C\notin \vvv_n\}. \]
It is easy to check that $\ccc_n$ is closed and $\ccc_n \cap \hhh=\varnothing$. So if we set
$\uuu_n=K(X)\setminus \ccc_n$ we see that the sequence $(\uuu_n)$ has all the desired properties.
\end{proof}
In the sequel we will say that the sequence $(\uuu_n)$ obtained by Lemma \ref{l1}, is the
\textit{normal form} of $\hhh$. We need the following lemmas.
\begin{lem}
\label{l2} Let $\uuu\subseteq K(X)$ be open, dense and hereditary. Let also $x_0,...,x_n$ be
distinct points in $X$ and $r_1>0$. Then there exist $y_0,...,y_n$ distinct points in $X$ such
that $d(x_i,y_i)<r_1$ for all $i\in\{0,...,n\}$ and moreover $\{y_0,...,y_n\}\in \uuu$.
\end{lem}
\begin{proof}
We may assume that $B(x_i,r_1)\cap B(x_j,r_1)=\varnothing$ for all $i,j\in \{0,...,n\}$ with
$i\neq j$. Let
\[ \vvv=\big\{ K: K\subseteq \bigcup_{i=0}^n B(x_i,r_1)\text{ and }
K\cap B(x_i,r_1) \neq\varnothing \ \forall i=0,...,n \big\}.\]
Then $\vvv$ is open. As $\uuu$ is open and dense, there exists $K\in\vvv\cap\uuu$.
For every $i\in\{0,...,n\}$ we select $y_i\in K\cap B(x_i,r_1)$. As $\uuu$ is hereditary,
we see that $\{y_0,...,y_n\}\in\uuu$. Clearly $y_0,...,y_n$ are as desired.
\end{proof}
\begin{lem}
\label{l3} Let $\uuu\subseteq K(X)$ be open, dense and hereditary. Let also $\ee>0$. Then the set
\[ G_{\uuu,\ee}=\{\mu\in P(X): \exists K\in\uuu \text{ with } \mu(K)\geq 1-\ee \}\]
is co-meager in $P(X)$.
\end{lem}
\begin{proof} Fix $\uuu$ and $\ee>0$ as above. We will show that for every $V\subseteq P(X)$ open
there exists $W\subseteq V$ open such that $W\subseteq G_{\uuu,\ee}$. This will finish the proof
(actually it implies that $G_{\uuu,\ee}$ contains a dense open set). So let $V\subseteq P(X)$
open. As finitely supported measures are dense in $P(X)$, we may select
$\nu=\sum_{i=0}^n a_i \delta_{x_i}$ and $r>0$ such that
\begin{enumerate}
\item[(1)] $a_i>0$ for all $i\in\{0,...,n\}$ and $\sum_{i=0}^n a_i=1$,
\item[(2)] $B_P(\nu,r)\subseteq V$.
\end{enumerate}
By Lemma \ref{l2}, there exist $y_0,...,y_n$ distinct points in $X$ with $\{y_0,...,y_n\}\in \uuu$
and such that $d(x_i,y_i)<\frac{r}{2}$ for all $i\in\{0,...,n\}$. We set
$\mu=\sum_{i=0}^n a_i \delta_{y_i}$. Then it is easy to see that
\begin{enumerate}
\item[(3)] $\rho(\mu,\nu)\leq \frac{r}{2}$.
\end{enumerate}
Let $F=\{y_0,...,y_n\}$. As $\uuu$ is open and $F\in\uuu$ there exists $\theta>0$ such that
\begin{enumerate}
\item[(4)] $\theta<\min\{\frac{\ee}{3},\frac{r}{3}\}$ and
\item[(5)] $B_H(F,2\theta)\subseteq \uuu$.
\end{enumerate}
Then $W=B_P(\mu,\theta)$ is as desired. Indeed, by (2), (3) and (4) it is clear that $W$ is a
subset of $V$. We only need to check that $W$ is a subset of $G_{\uuu,\ee}$. Let $\lambda\in W$
arbitrary. Then $\rho(\lambda,\mu)<\theta$ and so
\[ 1=\mu(F)\leq \lambda(F_{\theta}) +\theta \]
which gives that $\lambda(F_\theta)\geq 1-\frac{\ee}{3}$ by the choice of $\theta$. By the inner
regularity of $\lambda$, there exists $C\subseteq F_\theta$ compact such that
$\lambda(C)\geq 1-\ee$. We set $K=C\cup F$. Then $d_H(K,F)\leq\theta$ and so, by (5), $K\in \uuu$.
Moreover, $\lambda(K)\geq \lambda(C)\geq 1-\ee$. This implies that $\lambda\in G_{\uuu,\ee}$ and
the proof is completed.
\end{proof}
Our goal in this section is to prove the following.
\begin{prop}
\label{p1} Let $\hhh$ be a hereditary, dense $G_\delta$ subset of $K(X)$. Then the set
\[ G_\hhh=\{ \mu\in P(X): \forall \ee>0 \ \exists K\in \hhh \text{ with } \mu(K)\geq 1-\ee \} \]
is co-meager in $P(X)$.
\end{prop}
\begin{proof}
Let $(\uuu_n)$ be the normal form of $\hhh$. For every $n,m\in \nn$ let
\[ G_{n,m}=\big\{ \mu\in P(X): \exists K\in \uuu_n \text{ with } \mu(K)\geq 1-\frac{1}{m+1}\big\}.\]
By Lemma \ref{l3}, we have that $G_{n,m}$ is co-meager. Hence, so is $\bigcap_{n,m} G_{n,m}$.
We claim that $G_\hhh=\bigcap_{n,m} G_{n,m}$. This will finish the proof. It is clear that
$G_\hhh\subseteq\bigcap_{n,m} G_{n,m}$. Conversely, fix $\mu\in \bigcap_{n,m} G_{n,m}$ and let
$\ee>0$ arbitrary. Pick a sequence $(\ee_n)$ of positive reals such that
\[ \sum_{n\in\nn} \ee_n<\frac{\ee}{2}. \]
Pick also a sequence $(m_n)$ of natural numbers with $\frac{1}{m_n+1}\leq\ee_n$ for every
$n\in\nn$. As
\[ \mu\in \bigcap_{n,m} G_{n,m} \subseteq \bigcap_n G_{n,{m_n}}\]
we may select a sequence $(K_n)$ in $K(X)$ such that
\begin{enumerate}
\item[(1)] $K_n\in \uuu_n$ and
\item[(2)] $\mu(K_n)\geq 1-\frac{1}{m_n+1}\geq 1-\ee_n$.
\end{enumerate}
For every $n\in\nn$ we let $F_n= \bigcap_{i=0}^n K_i$ and we set $F=\bigcap_n K_n$.
Then $F_n\downarrow F$. Notice that $F\in \uuu_n$ as $F\subseteq F_n\subseteq K_n\in\uuu_n$
and $\uuu_n$ is hereditary. Hence $F\in \bigcap_n \uuu_n=\hhh$. Moreover, by (2) above, we have
\[ \mu(F_n)=\mu(K_0\cap ... \cap K_n)\geq 1- \sum_{k=0}^n \ee_k. \]
As $F_n\downarrow F$ we get that
\[ \mu(F)=\lim_{n\in\nn} \mu(F_n)\geq 1-\sum_{n\in\nn} \ee_n \geq 1-\ee. \]
This shows that $\mu\in G_\hhh$, as desired.
\end{proof}


\section{Left Haar-null sets in Polish groups}

Our aim is to give the proof of Theorem A stated in the introduction.
\begin{proof}[Proof of Theorem A]
Let $G$ be an uncountable Polish group and $A$ be a universally measurable subset of $G$ such that
$A^{-1}A$ is meager. We select a sequence $(C_n)$ of closed, nowhere dense subsets of $G$ with the
following properties.
\begin{enumerate}
\item[(i)] $1\notin C_n$ for all $n\in\nn$.
\item[(ii)] $A^{-1}A\setminus\{1\} \subseteq \bigcup_n C_n$.
\end{enumerate}
For every $n\in\nn$ let
\[ \uuu_n=\{ K\in K(G): K^{-1}K \cap C_n=\varnothing \}. \]
Clearly every $\uuu_n$ is hereditary. Moreover, as the function $f:K(G)\to K(G)$
defined by $f(K)=K^{-1}K$ is continuous, we see that every $\uuu_n$ is open.
\begin{claim}
\label{cl1} For every $n\in\nn$ the set $\uuu_n$ is dense in $K(G)$.
\end{claim}
\begin{proof}[Proof of Claim \ref{cl1}]
As finite sets are dense in $K(G)$, it is enough to show that for every finite subset
$\{x_0,...,x_l\}$ of $G$ and every $r>0$ there exist $y_0,...,y_l$ distinct points in $G$ with
\[\big\{ y_i^{-1}y_j: i,j\in\{0,...,l\} \text{ with } i\neq j\big\} \cap C_n=\varnothing\]
and such that $d(x_i,y_i)\leq r$ for all $i\in\{0,...,l\}$ (here $d$ is simply a compatible
complete metric of $G$). The points $y_0,...,y_l$ will be chosen by recursion. We set $y_0=x_0$.
Assume that $y_0,...,y_k$ have been chosen for some $k<l$ so as $\big\{ y_i^{-1}y_j:
i,j\in\{0,...,k\} \text{ with } i\neq j\big\} \cap C_n=\varnothing$. For every $g\in G$ the
functions $x\mapsto gx^{-1}$ and $x\mapsto gx$ are homeomorphisms. It follows that the set
$F_k=\bigcup_{i=0}^k (y_iC_n^{-1}\cup y_iC_n)$ is a closed set with empty interior. Hence there
exists $y_{k+1}\in B(x_{k+1},r)$ such that $y_{k+1}\notin F_k\cup\{ y_0,...,y_k\}$. This implies
that for every $i\in\{0,...,k\}$ we have $y_{k+1}^{-1}y_i\notin C_n$ and $y_i^{-1}y_{k+1}\notin
C_n$. This completes the recursive selection and the proof of the claim is completed.
\end{proof}
It follows by the above claim that the set $\hhh=\bigcap_n \uuu_n$ is a hereditary, dense
$G_\delta$ subset of $K(G)$ and that $(\uuu_n)$ is a normal form of $\hhh$. Notice that if
$K\in \hhh$, then $K^{-1}K \cap A^{-1}A =\{1\}$. By Proposition \ref{p1}, we have that the set
\[ B_1=\{ \mu\in P(G): \forall \ee>0 \ \exists K\in\hhh \text{ with } \mu(K)\geq 1-\ee \} \]
is co-meager. Our assumption that $G$ is uncountable implies that the Polish group $G$ viewed
as a topological space is perfect. Hence, the set of all non-atomic Borel probability measures
on $G$ is co-meager in $P(G)$ (see \cite{Kn}, or \cite{PRV}). It follows that the set
\[ B_2 =\{ \mu\in P(G): \mu \text{ is non-atomic and } \mu\in B_1\}\]
is co-meager in $P(G)$. We will show that $B_2\subseteq T_l(A)$. This will finish the proof.
We need the following fact (its easy proof is left to the reader).
\begin{fact}
Let $\mu\in P(G)$. Then $\mu\in T_l(A)$ if and only if for every $\nu\in P(G)$ we have
$\nu*\mu(A)=0$.
\end{fact}
Fix $\mu\in B_2$. By the above fact, in order to verify that $\mu\in T_l(A)$ we have to show that
$\nu*\mu(A)=0$ for every $\nu\in P(G)$. So, let $\nu\in P(G)$ arbitrary. Let also $\ee>0$
arbitrary. As $\mu\in B_2\subseteq B_1$, there exists $K\in\hhh$ with $\mu(K)\geq 1-\ee$. Then
\begin{eqnarray*}
\nu*\mu(A) & = & \int_G \nu(Ay^{-1}) d\mu(y) \leq \int_K \nu(Ay^{-1})d\mu(y) + \mu(G\setminus K)\\
& \leq & \int_K \nu(Ay^{-1})d\mu(y) + \ee.
\end{eqnarray*}
We set $I=\{ y\in K: \nu(Ay^{-1})>0\}$.
\begin{claim}
\label{cl2} The set $I$ is countable.
\end{claim}
\begin{proof}[Proof of Claim \ref{cl2}]
Notice that if $y, z\in I$ with $y\neq z$, then $Ay^{-1}\cap Az^{-1}=\varnothing$. For if not,
we would have that $1\neq y^{-1}z \in K^{-1}K\cap A^{-1}A$, which contradicts the fact that
$K\in \hhh$. It follows that the family $\{ Ay^{-1}: y\in I\}$ is a family of pairwise disjoint
sets of positive $\nu$-measure. Hence $I$ is countable, as claimed.
\end{proof}
The measure $\mu$ is non-atomic as $\mu\in B_2$. Hence, by Claim \ref{cl2}, we see that
$\mu(I)=0$. It follows that
\[ \int_K \nu(Ay^{-1}) d\mu(y)= \int_I \nu(Ay^{-1}) d\mu(y)\leq\mu(I)=0\]
and so $\nu*\mu(A)\leq\ee$. Since $\ee$ was arbitrary, this implies that $\nu*\mu(A)=0$.
The proof of Theorem A is completed.
\end{proof}
Combining Theorem A with Pettis' Theorem (see \cite[Theorem 9.9]{Kechris}) we get the following
corollary.
\begin{cor}
\label{c1} Let $G$ be an uncountable Polish group and $A$ an analytic subset of $G$.
If $A$ is not generically left Haar-null (in particular, if $A$ is not left Haar-null),
then $1\in\inte(A^{-1}AA^{-1}A)$.
\end{cor}
Clearly Theorem A implies that in non-locally-compact groups, compact sets are generically left
Haar-null. Another application of this form concerns the size of analytic subgroups of Polish
groups. Specifically we have the following corollary which may be considered as the
non-locally-compact analogue of M. Laczkovich's Theorem \cite{La}.
\begin{cor}
\label{c2} Let $G$ be an uncountable Polish group and $H$ be an analytic subgroup of $G$ with
empty interior. Then $H$ is generically left Haar-null.
\end{cor}
\noindent \textbf{What about Haar-null sets?} We would like to remark on the possibility of
extending Theorem A to Haar-null sets instead of merely left Haar-null. As it has been shown
by S. Solecki in \cite{S2}, the Steinhaus property of the $\sigma$-ideal of Haar-null sets fails
in a large number of Polish groups (in a sense, it fails for most non-abelian Polish groups).
Precisely, by \cite[Theorem 6.1]{S2}, if $(H_n)$ is a sequence of countable groups such that
infinitely many of them are not FC (see \cite{S2} for the definition of FC groups), then one can
find a closed set $A\subseteq \prod_n H_n$ which is not Haar-null and $A^{-1}A$ is meager. So,
there is no analogue of Theorem A for Haar-null sets in arbitrary Polish groups. Yet there is one
if we further assume that the group $G$ satisfies the following non-singularity condition.
\begin{enumerate}
\item[(C)] For every analytic and meager subset $A$ of $G$, the conjugate saturation
$[A]=\{x: \exists g\in G \ \exists a\in A \text{ with } x=gag^{-1} \}$ of $A$ is meager.
\end{enumerate}
Clearly every abelian Polish group satisfies (C). Moreover we have the following.
\begin{prop}
\label{p2} Let $G_1$ and $G_2$ be Polish groups. If both $G_1$ and $G_2$ satisfy (C), then so
does $G_1\times G_2$.
\end{prop}
\begin{proof}
Let $A\subseteq G_1\times G_2$ be analytic and meager. By the Kuratowski-Ulam theorem (see
\cite[Theorem 8.41]{Kechris}), we have
\[ \forall^* x\in G_1 \text{ the section } A_x=\{y\in G_2: (x,y)\in A\} \text{ of } A
\text{ is meager.} \]
As $G_2$ satisfies (C), by another application of the Kuratowski-Ulam theorem we get that
\[ A_1=\{ (x,z): \exists g_2, y\in G_2 \text{ with } (x,y)\in A \text{ and } y=g_2zg_2^{-1} \}\]
is analytic and meager. With the same reasoning we see that the set
\[ A_2=\{ (w,z): \exists g_1, x\in G_1 \text{ with } (x,z)\in A_1 \text{ and } x=g_1wg_1^{-1} \}\]
is analytic and meager too. Noticing that $A_2=[A]$, the result follows.
\end{proof}
For groups that satisfy (C) we have the following strengthening of Theorem A.
\begin{prop}
\label{p3} Let $G$ be an uncountable Polish group that satisfies (C). If $A$ is an analytic
subset of $G$ such that $A^{-1}A$ is meager, then $T(A)$ is co-meager.
\end{prop}
\begin{proof}
The proof is similar to the proof of Theorem A, and so, we shall only indicate the necessary
changes. Let $A\subseteq G$ be analytic such that $A^{-1}A$ is meager. Notice that $A^{-1}A$
is analytic. The group $G$ satisfies (C). It follows that the set $[A^{-1}A]$ is meager too.
Arguing as in the proof of Theorem A this implies that there exists a co-meager set $B_2$ of
non-atomic Borel probability measures on $G$ such that for every $\mu\in B_2$ and every $\ee>0$
there exists $K\subseteq G$ compact with $\mu(K)\geq 1-\ee$ and $K^{-1}K\cap [A^{-1}A]=\{1\}$.
We claim that $B_2\subseteq T(A)$. To this end, it is enough to show that for every $\mu\in B_2$,
every $\nu\in P(G)$ and every $x\in G$ we have $\nu*\mu(Ax)=0$. Let $\ee>0$ arbitrary and pick
$K\subseteq G$ compact as described above. Then
\[ \nu*\mu(Ax)\leq \int_K \nu(Axy^{-1}) d\mu(y) + \ee. \]
We set $I=\{y\in K: \nu(Axy^{-1})>0\}$. Observe that if $y, z\in I$ with $y\neq z$,
then $(Axy^{-1})\cap (Axz^{-1})=\varnothing$ (for if not, we would have that
$1\neq y^{-1}z\in K^{-1}K \cap [A^{-1}A]$). By the countable chain condition of $\nu$, we get
that $I$ is countable and the result follows.
\end{proof}
\begin{rem}
The $\sigma$-ideal of generically left Haar-null sets is a quite satisfactory $\sigma$-ideal of
measure-theoretic small sets in arbitrary Polish groups. Beside Theorem A, this is also supported
by the results in \cite{D1} asserting that every analytic and generically left Haar-null subset
$A$ of $G$ can be covered by a \textit{Borel} set $B$ with the same property. The fact that this
ideal is well-behaved is also reflected in the complexity of the collection of all closed
generically left Haar-null sets (in the Effros-Borel structure). It is much better than the
one of closed Haar-null sets, at least in abelian Polish groups. Specifically, it follows by the
results of S. Solecki in \cite{S1}, that in non-locally-compact abelian Polish groups the
$\sigma$-ideal of closed generically Haar-null sets is $\PB^1_1$-complete. The corresponding
collection of closed Haar-null sets is much more complicated (it is both $\SB^1_1$ and
$\PB^1_1$-hard).
\end{rem}


\end{document}